\documentclass[12pt,reqno,twoside]{amsart}

\usepackage{amssymb,amsmath,amscd,enumerate,verbatim}
\usepackage[colorlinks=true,linkcolor=blue,citecolor=blue]{hyperref}
\usepackage{amsmath}
\usepackage[latin1]{inputenc}
\usepackage{color}

\label{key}
\usepackage[all]{xy}
\usepackage{pstricks, pst-3d}

\newcommand{\mm}{\mathfrak m}
\newcommand{\nn}{\mathfrak n}

\newcommand{\Dsf}{\mathsf{D}}

\newcommand{\Z}{\mathbb{Z}}

\newcommand{\N}{\mathbb{N}}

\newcommand{\Fc}{\mathcal{F}}

\newcommand{\up}[1]{{{}^{#1}\!}}

\newcommand\bst{{\boldsymbol t}}
\newcommand\bsu{{\boldsymbol u}}
\newcommand\bsv{{\boldsymbol v}}

\newcommand\bsx{{\boldsymbol x}}
\newcommand\bsy{{\boldsymbol y}}
\newcommand\bsz{{\boldsymbol z}}
\DeclareMathOperator{\pnt}{\raise 0.5mm \hbox{\large\bf.}}

\DeclareMathOperator{\Coker}{Coker}

\DeclareMathOperator{\depth}{depth}

\DeclareMathOperator{\gr}{gr}
\DeclareMathOperator{\chara}{char}
\DeclareMathOperator{\Tor}{Tor}

\DeclareMathOperator{\Ext}{Ext}

\DeclareMathOperator{\Ann}{Ann}
\DeclareMathOperator{\lind}{ld}

\DeclareMathOperator{\linp}{lin}
\DeclareMathOperator{\Ker}{Ker}

\DeclareMathOperator{\id}{id}
\DeclareMathOperator{\img}{Im}
\DeclareMathOperator{\reg}{reg}

\DeclareMathOperator{\order}{order}
\DeclareMathOperator{\projdim}{pd}

\def\+#1{\relax\ifmmode\if\noexpand #1\relax \mathop{\kern
    0pt^+{#1}}\nolimits\else \kern 0pt^+\!#1 \fi\else$^*$#1\fi}

\newcommand{\dcat}[1]{{\mathsf D}(#1)}
\newcommand{\dgrcat}[1]{{\mathsf D}^{\mathsf f}(\mathsf{Gr}\,#1)}

\newcommand{\dgrcatn}[1]{{\mathsf{D}^{\mathsf f}_{\!{\scriptscriptstyle\mathsf -}}}(\mathsf{Gr}\,#1)}
\newcommand{\dgrcatp}[1]{{\mathsf{D}^{\mathsf f}_{\!{\scriptscriptstyle\mathsf +}}}(\mathsf{Gr}\,#1)}
\newcommand{\dtensor}[1]{\otimes^{\mathsf L}_{#1}}
\newcommand{\grcat}[1]{{\mathsf{Gr}}\,#1}
\newcommand{\shift}{\mathsf{\Sigma}}
\newcommand{\vphi}{\varphi}

\let\:=\colon

\newtheorem{thm}{\bf Theorem}[section]
\newtheorem{lem}[thm]{\bf Lemma}
\newtheorem{cor}[thm]{\bf Corollary}
\newtheorem{prop}[thm]{\bf Proposition}

\newtheorem{quest}[thm]{\bf Question}

\theoremstyle{definition}
\newtheorem{defn}[thm]{\bf Definition}

\theoremstyle{plain}
\newtheorem*{thm*}{Theorem}

\theoremstyle{remark}
\newtheorem{rem}[thm]{Remark}

\newtheorem{ex}[thm]{Example}

\numberwithin{equation}{section}

\title[Koszulness and the Frobenius morphism]{Regularity over homomorphisms and a Frobenius characterization of Koszul algebras}
\author{Hop D. Nguyen}
\address{Dipartimento di Matematica, Universit\`a di Genova, Via Dodecaneso 35, 16146 Genoa, Italy}
\address{Institut f\"ur Mathematik, Friedrich-Schiller Universit\"at Jena, 
Ernst-Abbe-Platz 2, 07743 Jena}
\email{ngdhop@gmail.com}
\author{Thanh Vu}
\address{Department of Mathematics, University of California at Berkeley, 
Berkeley CA 94720}
\address{Department of Mathematics, University of Nebraska-Lincoln, Lincoln, 
NE 68588}
\email{tvu@unl.edu}

\begin{document}

\begin{abstract}
Let $R$ be a standard graded algebra over an $F$-finite field of characteristic 
$p > 0$. Let $\phi:R\to R$ be the Frobenius endomorphism. For each finitely 
generated graded $R$-module $M$, let $\up{\phi}M$ be the abelian group $M$ with 
the $R$-module structure induced by the Frobenius endomorphism. The $R$-module 
$\up{\phi}M$ has a natural grading given by $\deg x=j$ if $x\in M_{jp+i}$ for 
some $0\le i \le p-1$. In this paper, we prove that $R$ is Koszul if and only if 
there exists a non-zero finitely generated graded $R$-module $M$ such that 
$\reg_R \up{\phi}M <\infty$. This result supplies another instance for the 
ability of the Frobenius in detecting homological properties, as exemplified by 
Kunz's famous regularity criterion. The main technical tool is the notion of 
Castelnuovo-Mumford regularity over certain homomorphisms between $\N$-graded 
algebras. The latter notion is a common generalization of the relative and 
absolute Castelnuovo-Mumford regularity of modules.
\end{abstract}

\maketitle

\section{Introduction}
\label{intro}
Let $R$ be a commutative noetherian ring and $M$ a finitely generated 
$R$-module. In commutative algebra, it is well-known that the homological 
properties of $R$ greatly affect the asymptotic homological behaviour of $M$. 
For example, if $R$ is regular local then the Auslander-Buchsbaum-Serre theorem 
says that the minimal free resolution $M$ terminates after finitely many steps. 
For another example, consider Koszul algebras. Assume that $R$ is a graded 
algebra over a field $k$ and $R$ is generated by finitely many elements of 
degree $1$ (in other words, $R$ is standard graded over $k$). The graded maximal 
ideal of $R$ is denoted by $\mm$. Then $R$ is a {\em Koszul algebra} if the 
Castelnuovo-Mumford regularity of $k=R/\mm$ is zero. By definition, if $M$ is a 
finitely generated graded $R$-module with the minimal graded free resolution
\[
\cdots \to F_i \to F_{i-1} \to \cdots \to F_0 \to M \to 0,
\]
then the {\em Castelnuovo-Mumford regularity of $M$} as an $R$-module is defined 
as follows
\[
\reg_R M=\sup\{t_i(M)-i:i\ge 0\},
\]  
where for each $i$, $t_i(M)$ is the maximal degree of a minimal generator of $F_i$. In closed form, $t_i(M)=\sup\{j:\Tor^R_i(M, R/\mm)_j\neq 0\}$. When $R$ is a Koszul algebra, Avramov and Eisenbud in \cite{AE} showed that $\reg_R M < \infty$ for every $M$. For a recent survey on Koszul algebras and related topics, we refer to \cite{CDR}. 

Conversely and perhaps more surprisingly, the asymptotic homological behaviour 
of {\it a single module} $M$ may force certain homological properties on $R$. 
For example, take $M=k$ the residue field; then Avramov and Peeva \cite{AP} 
showed that the finiteness of $\reg_R k$ implies the Koszul property of $R$. 
This is an analog of Auslander-Buchsbaum-Serre's characterization of regular 
local rings. Further results of this type were also discovered for, e.g., 
complete intersections, Gorenstein rings or Cohen-Macaulay rings. 

But the residue field is not the only module for detecting homological 
properties of $R$. For local rings of characteristic $p > 0$, results of Kunz 
\cite{K} and more generally of Rodicio \cite{Rod} showed that $\up{\phi}R$ is 
equally good as the residue field as a testing object for the regularity 
property of $R$, where $\phi$ is the Frobenius endomorphism. Recall that, if 
$R$ is a ring of characteristic $p>0$, the Frobenius endomorphism $\phi: 
R\to R$ is given by $\phi(a)=a^p$ for each $a\in R$. Given $e\in \Z, e\ge 1$, 
let $\phi^e$ be the $e$-th power of $\phi$. Denote by $\up{\phi^e}M$ the abelian 
group $M$, considered as an $R$-module via scalar restriction along $\phi^e$. If 
$R$ is local, Kunz \cite{K} proved that $R$ is a regular ring if and 
only if $\up{\phi^e}R$ is a flat $R$-module. Built on Kunz's criterion, a 
result of Rodicio \cite[Theorem 2]{Rod} says that a local ring $R$ 
is regular if and only if $\up{\phi}R$ has finite flat dimension as an 
$R$-module. For further information about the homological significance of the 
Frobenius endomorphism, the reader may consult, e.g., \cite{AIM}, \cite{BM}, 
\cite{IS}, \cite{Mi}, \cite{TY}. Our initial goal was to study an analog of 
Rodicio's theorem for Koszul algebras. 

To set up the result, we need first to equip gradings to modules with the 
$R$-action induced by the Frobenius. If $R$ is standard graded and $M$ is 
graded, the module $\up{\phi^e}M$ has a natural grading as follows. Denote 
$q=p^e$. Observe that $\phi^e$ is also a homogeneous ring homomorphism from $R$ 
to $R^{(q)}=\oplus_{j\ge 0}R_{jq}$, the $q$-th Veronese subring of $R$. For each 
$i=0,1,\ldots,q-1$, the Veronese module $V_i(q,M)=\bigoplus_{j\in \Z}M_{jq+i}$ 
comes equipped with the grading $\deg x=j$ if $x\in M_{jq+i}$, as a module over 
$R^{(q)}$. By scalar restriction along $\phi^e$, each $V_i(q,M)$ becomes a 
graded $R$-module, and so does $\up{\phi^e}M=\bigoplus_{i=0}^{q-1}V_i(q,M)$. We 
call this grading of $\up{\phi^e}M$ the {\em Veronese grading}. 

Recently, in \cite{AHIY}, Avramov, Hochster, Iyengar and Yao gave a vast 
generalization of Rodicio's result. Namely, they showed in 
\cite[Theorem~1.1]{AHIY} that for a local ring $R$ with $\chara R>0$, $R$ is 
regular if there exist an $e>0$ and a non-zero finitely generated $R$-module $M$ 
such that $\up{\phi^e}M$ has finite flat dimension. Suggested by this result, we 
are able to prove the following analog of Rodicio's result, namely a Frobenius 
characterization of Koszul algebras. In the statement, recall that $R$ is called 
{\em $F$-finite} if $\phi:R\to R$ is a finite morphism (for example $\Z/(p)$ is 
an $F$-finite field).  
\begin{thm}
\label{main1}
Let $k$ be an $F$-finite field of characteristic $p>0$ and $R$ be a standard graded $k$-algebra. Assume that there exist an $e>0$ and a non-zero finitely generated graded $R$-module $M$ such that $\reg_R \up{\phi^e}M <\infty$. Then $R$ is a Koszul algebra.
\end{thm}

If $k$ is $F$-finite, it is easy to show that $\phi^e$ is a finite endomorphism 
for every $e\ge 1$; see Remark \ref{Frobenius_finite}. In this paper, we will 
also improve Theorem \ref{main1} by allowing $k$ to be of arbitrary 
characteristic and replacing the power $\phi^e$ of Frobenius by a certain 
(possibly) non-finite endomorphism of $R$. To afford such a statement, for any 
homomorphism of graded algebras $\vphi:R\to S$ and any complex of graded modules 
$M$ over $S$ (satisfying mild conditions), we introduce the {\em regularity of 
$M$ over the homomorphism $\vphi$}, denoted by $\reg_{\vphi}M$. The regularity 
over a homomorphism is a broad generalization of the Castelnuovo-Mumford 
regularity. It will be constructed in Section \ref{reg_over_morphism} using 
ideas of Avramov, Iyengar and Miller \cite{AIM}, which were later developed in 
\cite{AHIY}. The main technical result in this paper, which generalizes Theorem 
\ref{main1}, is:
\begin{thm}
\label{orderge2}
Let $R$ be a standard graded $k$-algebra. Let $\psi: R\to R$ be an endomorphism {\em of order $d\ge 2$}, i.e. $\deg \psi(a)=d\deg a$ for any homogeneous element $a\in R$. Assume that there exists a non-zero finitely generated graded $R$-module $M$  such that $\reg_{\psi}M<\infty$. Then $R$ is a Koszul algebra.
\end{thm}
See Remark \ref{char_0} for examples of endomorphisms of order at least $2$ in characteristic $0$.

Before going into the details, we would like to mention three features in the 
proofs of Theorems \ref{main1} and  \ref{orderge2}. Firstly, although the 
statements of these results do not involve complexes of modules, their proofs 
use heavily the derived categories and homological algebra of complexes. The 
necessary background materials will be presented in Section \ref{background}. 
Secondly, to derive the main Theorem \ref{main1} from Theorem \ref{orderge2}, 
we need to 
compare the regularity of $\up{\phi^e}M$ over $R$ with the regularity of $M$ 
over the homomorphism $\phi^e$. The foundation for this comparison is laid in 
Sections \ref{Veronese_gradings} (on Veronese gradings) 
and \ref{reg_over_morphism} and its goal will be reached in Theorem 
\ref{reg_fract_Veronese}(ii). Finally, the proof of Theorem \ref{orderge2} is 
based on passing from the morphism $\psi: R\to R$ to one of its iterations. The 
passage is possible thanks to results in Section 
\ref{factorizations_compositions}, particularly Theorem \ref{composition} on 
regularity over compositions of homomorphisms. The fact that $\psi$ has order at 
least $2$ allows the passage from the module $M$ to certain complex over a 
Veronese subring of $R$. After these modifications of morphism and module, we 
use Theorem \ref{relative_thm} to conclude that $R$ is Koszul. The latter result 
was inspired by Apassov's work \cite[Theorem~R]{Apa}. Below, the condition 
``$M\in \dgrcatp S$" means that the homology $H_i(M)$ is finitely generated for 
each $i$ and $H_i(M)=0$ for $i\ll 0$ (see Section \ref{background}).
\begin{thm}
\label{relative_thm}
Let $\vphi: R\to S$ be a homogeneous morphism of standard graded algebras over the fields $k$ and $l$ respectively, where $S$ is a Koszul $l$-algebra. Assume that there exists a complex of graded $S$-modules $M\in \dgrcatp S$ such that $M\not\simeq 0$ and $\reg_{\vphi}M<\infty$. Then $R$ is a Koszul $k$-algebra.
\end{thm}
The proofs of the main results can be found in Section \ref{F_application}. In the final section of this paper, we provide a partial analog of Theorem \ref{main1} for local rings and raise some related questions.

\subsection*{Acknowledgements}
We wish to express our special gratitude to Srikanth B. Iyengar for his critical 
comments and suggestions. His suggestion of Proposition \ref{reg_of_complex} 
helped us to correct an error. More importantly, he brought to our attention the 
method in \cite{AHIY} and a rough version of Theorem \ref{composition}, all of 
which were decisive for us to arrive at the main results in this paper.

We are indebted to the two anonymous referees for their attentive reading of 
the paper and valuable criticism. Thanks to their excellent editorial work, 
many mathematical inaccuracies were cleared up and the quality of the 
presentation including subtle notational details has been tremendously improved.

The first named author was supported by the CARIGE foundation.

\section{Preliminaries}
\label{background}
\subsection*{Graded complexes}
Let $(R,\mm,k)$ be an $\N$-graded $k$-algebra with graded maximal ideal $\mm$, 
where $k$ is a field. By convention, all the algebras in this paper are finitely 
generated. Modules are identified with chain complexes concentrated in degree 
$0$. Throughout, complexes are chain complexes. The abelian category of 
complexes of graded modules and degree $0$ homomorphisms of complexes over $R$ 
is denoted by $\grcat R$. Let $\Dsf(\grcat R)$ be its derived category. Let 
$\dgrcat R$ be the full subcategory of $\Dsf(\grcat R)$ consisting of 
degreewise finite complexes, i.e.~complexes $M$ such that $H_i(M)$ is finitely 
generated for every $i$. We say that $M$ is {\em homologically bounded below} 
(or above) if $H_i(M)=0$ for $i\ll 0$ (respectively, for $i\gg 0$). The full 
subcategory of $\dgrcat R$ consisting of homologically bounded below complexes 
is denoted by $\dgrcatp R$. Similarly, the notation $\dgrcatn R$ stands for the 
full subcategory of $\dgrcat R$ consisting of homologically bounded above 
complexes.

We would like to warn the reader that the subscripts may represent either homological degree or internal degree. Namely, for a {\it complex} of graded $R$-modules $M$, $M_i$ denotes the module of $M$ at {\it homological degree} $i$. On the other hand, for a graded $R$-{\it module} $N$, $N_i$ denotes the graded component of elements of {\it internal degree} $i$ of $N$. It should be clear from the context how one should interpret the meaning of the subscript.

The notation $\simeq$ signifies isomorphisms in $\dgrcat R$. The notation $\shift$ denotes the functor shift for complexes: $(\shift M)_i=M_{i-1}$ for each $i\in \Z$. For a graded $R$-module $N$ and an integer $i$, the module $N(i)$ is given by: $N(i)_j=N_{i+j}$ for each $j$. More generally, for a {\em complex} of graded $R$-modules $M$ and each integer $i$, the complex $M(i)$ is obtained by shifting the internal degree of each module of $M$ suitably.

Denote by $\inf M$ the infimum of the set $\{i:H_i(M) \neq 0\}$.

For a complex $M\in \dgrcatp R$, let $G$ be the minimal graded free resolution of $M$ in $\dgrcat R$. Let $G_i=\oplus_j R(-j)^{\beta_{i,j}(M)}$ for each $i$. The numbers $\beta_{i,j}(M)$ are called the graded Betti numbers of $M$. The Poincar\'e series is a good way to encode information about the minimal free resolution of a complex. The Poincar\'e series of $M$ is the formal power series $P^R_M(t,y)=\sum_{i\in \Z}\sum_{j\in \Z} \beta_{i,j}(M)t^iy^j \in \Z[[t,y]]$. We refer to \cite{CF}, \cite{W} for more detailed treatments of homological algebra of complexes and to \cite{Jor1}, \cite{Jor2} for the graded setting. See also \cite{Avr} for an in-depth discussions of free resolutions.
\subsection*{Regularity}
Recall that if $R$ is standard graded, then the relative Castelnuovo-Mumford 
regularity {\em over $R$} of a complex $M\in \dgrcatp R$ is defined by
\[
\reg_R M = \sup\{j-i:\beta_{i,j}(M)\neq 0\}.
\]
In the literature, sometimes $\reg_R M$ is referred to as the 
$\Ext$-regularity; see for example \cite{Jor1}, \cite{Jor2}. We would like to 
emphasize that the notion of ``Castelnuovo-Mumford regularity'' defined in the 
above two papers is the {\em absolute} one, when the base ring is a commutative 
graded algebra (see Remark \ref{rem_compare}(ii)). However, most of the time in 
this paper, we will only deal with 
{\em relative} regularity, which generally only coincides with the absolute 
regularity for regular base rings.

The following simple lemma will be used several times in the sequel.
\begin{lem}
\label{reg_sequence}
Let $R$ be a standard graded $k$-algebra.
\begin{enumerate}[\quad \rm(i)]
\item If $0\to M' \to M \to M'' \to 0$ is an exact sequence of finitely generated graded $R$-modules, then there are inequalities
\begin{align*}
\reg_R M  &\le \max\{\reg_R M', \reg_R M''\},\\
\reg_R M' &\le \max\{\reg_R M, \reg_R M''+1\},\\
\reg_R M''&\le \max\{\reg_R M, \reg_R M'-1\}.
\end{align*}
\item \textnormal{(See \cite[Lemma~2.2(b)]{BCR}.)} For an exact sequence of finitely generated graded $R$-modules
$$
\cdots \to P_i \to P_{i-1} \to \cdots \to P_0 \to N \to 0,
$$
we have an inequality
\[
\reg_R N \le \sup\{\reg_R P_i-i: i\ge 0\}.
\]
\end{enumerate}
\end{lem}

\subsection*{Scalar restriction}
We record the following simple fact for later usage. Let $R\to S$ be a ring homomorphism. Every $S$-module is automatically an $R$-module via scalar restriction. Denote by $\dcat R$ the derived category of complexes of $R$-modules. We say that a functor between triangulated categories is exact (or triangulated) if it commutes with the translations and preserves the distinguished triangles (see \cite[Chapter 2]{Ne} for more details).
\begin{lem}
\label{scalar_restr}
The scalar restriction functor from $\dcat S$ to $\dcat R$ is an exact functor.
\end{lem}
\section{The Veronese grading and fractional Veronese}
\label{fract_Veronese}
When $\chara k=p>0$, and $R$ is an $\N$-graded $k$-algebra, the Frobenius $\phi:R\to R$ given by $\phi(a)=a^p$ is a ring homomorphism. However, it is neither $k$-linear nor homogeneous. There are two opposite ways to ``homogenize" the Frobenius: considering Veronese subrings of $R$, and considering the so-called fractional Veronese functor $(\cdot)^{(1/p)}$. To obtain more general results, we will work with a class of ring homomorphisms which strictly contains powers of the Frobenius, namely the homomorphisms with positive orders.
\subsection*{The Veronese grading}
The following type of homomorphisms will frequently appear in this paper.
\label{Veronese_gradings}
\begin{defn}
Let $\vphi: (R,\mm,k)\to (S,\nn,l)$ be a homomorphism of algebras such that $\deg \vphi(a)=d \deg(a)$ for every homogeneous element $a\in R$ where $d\ge 1$ is an integral constant. Then we say that $d$ is the {\em order} of $\vphi$ and write $\order \vphi=d$.
\end{defn}
For example, the power $\phi^e$ of Frobenius endomorphism has order $q=p^e$ for each $e\ge 1$. Homogeneous morphisms are simply those of order $1$. Clearly if $\vphi$ is a homomorphism of order $d$, then we have an induced {\em homogeneous} morphism $R\to S^{(d)}$.

Let $M$ be a graded $S$-module. Consider the $d$-th Veronese modules 
$V_i(d,M)=\oplus_{u\in \Z}M_{du+i}$ for $0\le i\le d-1$. These are graded 
modules over the Veronese ring $S^{(d)}$. We will still use $M$ to denote the 
graded $S^{(d)}$-module $\oplus_{i=0}^{d-1}V_i(d,M)$, which is the abelian 
group $M$ itself with a modified grading.

We denote by $\up{\vphi}M$ the {\em $R$-module} 
 $M=\oplus_{i=0}^{d-1}V_i(d,M)$ with the $R$-action induced by scalar 
restriction along $R\to S^{(d)}$. Hence the grading of $\up{\vphi}M$ is given by
$$
\deg M_{du+i}=u,
$$
and we call it the Veronese grading. In the case $\chara k=p>0$, $S=R$ and 
$\vphi=\phi^e$, we obtain the Veronese grading for $\up{\phi^e}M$ mentioned in 
the introduction.

\subsection*{Fractional Veronese}
Let $R$ be an $\N$-graded algebra over $k$, such that $R$ is generated in a single degree $g \ge 1$. Let $M$ be a graded $R$-module and $s\ge 1$ an integer. Define $R^{(1/s)}$ to be the ring $R$ with a new grading given by: $R^{(1/s)}_{si}=R_i$ for all $i\in \Z$ and $R^{(1/s)}_j=0$ for all $j$ not divisible by $s$. We call $R^{(1/s)}$ the {\it $s$-th fractional Veronese} of $R$. A similar definition holds for $M^{(1/s)}$, and the latter is an $R^{(1/s)}$-module. Furthermore, $R^{(1/s)}$ is a $k$-algebra generated in degree $sg$, and the fractional Veronese functor $(\cdot)^{(1/s)}$ from $\dgrcat R$ to $\dgrcat {R^{(1/s)}}$ is exact.

Endomorphisms of order $\ge 1$ and fractional Veronese are not new notions, for instances they appeared (without names) in the work of Koh and Lee \cite[p.~ 686]{KL}. Fractional Veronese keeps all the information about the ring $R$ and the module $M$, except for the grading. This is the difference between the fractional Veronese and the Veronese functor.

Note that if $\vphi: R\to S$ has order $d$, then one gets an induced homogeneous ring homomorphism $R^{(1/d)}\to S$, denoted by $\widetilde{\vphi}$. Any $S$-module is naturally a module over $R^{(1/d)}$. Given $d\ge 1$ and an endomorphism $\psi: R\to R$ of order $d$, any graded $R$-module is naturally an $R^{(1/d)}$-module via $\widetilde \psi$. We will construct a functor relating graded modules as well as regularity of complexes over the rings $R^{(1/d)}$ and $R$.
\begin{defn}
\label{defn_functor}
Let $N$ be a graded $R^{(1/d)}$-module. The $d$-th Veronese modules of $N$ are 
naturally graded {\em $R$-modules} via the induced action: if $a\in R$ and 
$x\in V_i(d,N)$ where $0\le i\le d-1$, we define $a\circ x=ax$, where the 
multiplication with scalar on the right-hand side is that of $R^{(1/d)}$. In 
fact, if $a\in R_s$ and $x\in N_{dr+i}$ (so $\deg x=r$ in $V_i(d,N)$), then 
$a\in R^{(1/d)}_{ds}$ and hence $ax \in N_{d(s+r)+i}$. Thus $\deg (a\circ 
x)=s+r$ in $V_i(d,N)$, as expected. We define $\Phi(N)$ to be the graded 
$R$-module $\oplus_{i=0}^{d-1}V_i(d,N)$ with the above $\circ$ action. Then 
define the image via $\Phi$ of  complexes of graded $R^{(1/d)}$-modules and 
morphisms between such complexes in a natural way.
\end{defn}
The important properties of $\Phi$ are stated in the following
\begin{lem}
\label{functor_varPhi}
The functor $\Phi$ from $\dgrcat {R^{(1/d)}}$ to $\dgrcat R$ is an exact functor. It is a left inverse of the fractional Veronese functor $(\cdot)^{(1/d)}$. Moreover, the $R$-module $\Phi(R^{(1/d)})$ has a natural structure of a graded $k$-algebra and there is a homogeneous isomorphism  $\Phi(R^{(1/d)})\cong R$ of graded algebras.
\end{lem}
\begin{proof}
That $\Phi$ is exact follows from Lemma \ref{scalar_restr} applied to the 
natural homomorphism $R\to R^{(1/d)}$. For any graded $R$-module $M$, one 
easily verifies that $\Phi(M^{(1/d)})_i=M_i$ for all $i\in \Z$. Hence $\Phi$ is 
a left inverse of $d$-th fractional Veronese.

Given $j\in \Z$, there are equalities
\[
\Phi(R^{(1/d)})_j= \oplus_{i=0}^{d-1}(R^{(1/d)})_{jd+i} = (R^{(1/d)})_{jd} =R_j.
\]
The $k$-algebra structure of $R^{(1/d)}$ naturally induces such a structure for $\Phi(R^{(1/d)})$. Hence we conclude that $\Phi(R^{(1/d)})\cong R$. Therefore the proof is completed.
\end{proof}
\section{Regularity over homomorphisms}
\label{reg_over_morphism}
In this section, we begin studying the notion of regularity over a homomorphism. Since the main applications we have in mind are for standard graded algebras in positive characteristics, we restrict ourself to working over $\N$-graded algebras which are generated in a single positive degree. We believe that it is possible to define regularity over homomorphisms for arbitrarily graded algebras, however that idea will not be pursued here. 
\begin{rem}
\label{Frobenius_finite}
Assume that $\chara k=p>0$ and $k$ is $F$-finite. For every $e>0$, $\phi^e$ is a finite morphism. Namely, if $R=k[x_1,\ldots,x_n]/I$ is a presentation of $R$, where $I$ is a homogeneous ideal of the polynomial ring $k[x_1,\ldots,x_n]$, then $\up{\phi^e}R$ is generated by the classes of the generators of $k$ as a $\phi^e(k)$-module  and the classes of $x_1^{a_1}\cdots x_n^{a_n}$, where $0\le a_j\le q-1$ for all $j=1,\ldots,n$.
\end{rem}
In this section, we will work with a (not necessarily finite) homomorphism $\vphi: (R,\mm,k)\to (S,\nn,l)$ of $\N$-graded rings with order $d\ge 1$. Assume that $R$ is generated in degree $g$ and $S$ is generated in degree $h$, where $g, h\ge 1$. Let $M\in \dgrcatp S$ be a complex of graded $S$-modules. 

\subsection*{Homotopy annihilator}
For our purpose, it is useful to recall the notion of the {\em homotopy annihilator} of a complex due to Apassov \cite{Apa2}.

Let $P$ be a free resolution of $M$. For any $a\in S$, the multiplication with $a$ induces an element of $\Ext^0_S(P,P)\cong \Ext^0_S(M,M)$. Composition of morphisms gives the latter the structure of a graded ring, therefore we have a homogeneous ring homomorphism $\eta_S: S\to \Ext^0_S(M,M)$. Denote by $\Ann_{\dcat{S}}M$ the homogeneous ideal $\Ker \eta_S$, which is called the homotopy annihilator of $M$. The homotopy annihilator is well-defined in the derived category $\dcat S$ (and in the graded category $\dgrcat S$). If $M$ is an $S$-module then $\Ann_{\dcat S}M$ is the usual annihilator ideal $\Ann_S M$ of $M$.

Since homotopic maps induce the same morphism on homology of a complex, we see that $\Ann_{\dcat{S}}M \subseteq \Ann_S H_i(M)$ for all $i\in \Z$. The proof of \cite[(1.5.6)]{AIM} applies verbatim to give the following result, even though our complex $M$ does not have bounded homology.
\begin{lem}
\label{lem_ann_dtensor}
Let $L\in \dgrcatp R$ be any complex and $\bsy$ a finite sequence of homogeneous elements in $S$. Then we have
\[
\bsy S +(\Ann_{\dcat R}L)S+\Ann_{\dcat S} M\subseteq \Ann_{\dcat S} (L\dtensor R 
K[\bsy;M]),
\]
where $K[\bsy;M]$ denotes the Koszul complex on $\bsy$ with coefficients in $M$.
\end{lem}

\subsection*{Regularity over homomorphisms}
Let us define Betti numbers and Poincar\'e series of $M$ over the homomorphism $\vphi$, following the construction in the local setting in \cite[Section 2]{AHIY}.

Recall that $\vphi$ induces a homogeneous homomorphism $\widetilde{\vphi}: 
R^{(1/d)} \to S$ and $R^{(1/d)}$ is generated in degree $dg$. Since $S$ is 
generated in degree $h$, we infer that $dg=hc$ for some integer $c\ge 1$. By 
abuse of notation, we write $\mm S$ for the extension ideal $\mm^{(1/d)}S$ with 
respect to the homomorphism $\widetilde{\vphi}$. By hypothesis, obviously $\mm 
S\subseteq \nn^c$, where $\nn^c\subseteq S$ denotes the $c$-th power of the 
ideal $\nn$. Note that $\nn^c$ is generated in degree $hc=dg$. Let 
$\bsx=x_1,\ldots,x_n$ be a minimal system of generators of the ideal $\nn^c$ 
modulo $\mm S$, such that $\deg x_i=dg$ for all $i=1,\ldots,n$. 

Let $K[\bsx;M] = K[\bsx;S] \dtensor {S} M$ denote the Koszul complex of $M$ with respect to the sequence $\bsx$ in $S$ (see \cite{AIM} for more information). For each $i$, $H_i(k\dtensor {R^{(1/d)}} K[\bsx;M])\cong \Tor^{R^{(1/d)}}_i(k,K[\bsx;M])$ is a finitely generated $S$-module. Moreover, by Lemma \ref{lem_ann_dtensor}, 
\[
\nn^c=\mm S+\bsx S \subseteq \Ann_{\dcat S} (k\dtensor {R^{(1/d)}} K[\bsx;M]),
\]
hence $H_i(k\dtensor {R^{(1/d)}} K[\bsx;M])$ is also an $S/\nn^c$-module. Therefore $\dim_l H_i(k\dtensor {R^{(1/d)}} K[\bsx;M])<\infty$ for all $i$. By definition, for each $i,j\in \Z$, the $(i,j)$ graded Betti number of $M$ over $\vphi$ is
\[
\beta^{\vphi}_{i,j}(M)=\dim_l H_i(k\dtensor {R^{(1/d)}} K[\bsx;M])_j.
\]
We will see that $\beta^{\vphi}_{i,j}(M)$ is well-defined in Proposition 
\ref{non_minimal_gen}. Given $i$, for all but finitely many $j$, we have 
$\beta^{\vphi}_{i,j}(M)=0$.
\begin{defn}
The {\em regularity of $M$ over the homomorphism $\vphi$}, denoted by 
$\reg_{\vphi}M$, is defined as follows:
\[
\reg_{\vphi}M= \sup\left\{\frac{j-idg}{dg}:\beta^{\vphi}_{i,j}(M)\neq 0 \right\}.
\]
\end{defn}
\begin{rem}
\label{rem_compare}
(i) If $\vphi=\id^R$ is the identity of $R$, we will denote $\reg_{\id^R}M$ simply by $\reg_R M$. It is clear from the definition that for every $m\in \Z$,
\begin{equation*}
\reg_R M(-m)=\reg_R M + \frac{m}{g}.
\end{equation*}
If additionally $R$ is standard graded then $\beta^{\vphi}_{i,j}(M)$ is the usual $(i,j)$-Betti number of $M$ over $R$ and $\reg_R M$ is the usual Castelnuo\-vo-Mumford regularity of $M$.

(ii) Let $(S,\nn)$ be a standard graded $l$-algebra and $M$ a finitely 
generated graded module over $S$. The absolute Castelnuovo-Mumford regularity 
of $M$ is defined by
$$
\reg M=\sup\{i+j:H^i_{\nn}(M)_j\neq 0\},
$$
where $H^i_{\nn}(M)$ denotes the $i$-th local cohomology of $M$ with support at 
$\nn$. Recall that if $Q$ is any polynomial ring over $l$ and $Q\to S$ is a
surjection of standard graded $l$-algebras, then $\reg M=\reg_Q M$. Let $\tau$ 
be the natural map $l\to S$, and $\bsx$ be a minimal system of generators of 
$\nn$. Then 
\[
\beta^{\tau}_{i,j}(M)=\dim_l H_i(K[\bsx;M])_j,
\]
for all $i,j$. This implies that $\reg_{\tau}M=\reg M$. Hence the absolute 
regularity is a special case of regularity over homomorphisms.

(iii) If $R, S$ are standard graded algebras and $\vphi: R\to S$ is a finite homogeneous morphism, then $\reg_{\vphi}M$ is equal to $\reg_R M$, the regularity of $M$ viewed as a complex of $R$-modules via scalar restriction; see Remark \ref{artinian_case} below. 
\end{rem}
The following proposition shows that graded Betti numbers and regularity of $M$ {\em over $\vphi$} do not depend on the choice of minimal generators of $\nn^c/\mm S$. Moreover, we can compute the regularity of $M$ over $\vphi$ by choosing any (minimal or not) generating set of $\nn^c$ modulo $\mm S$ consisting of elements of degree $dg$. In the result below, the equality between formal series is a graded analog of \cite[Proposition 4.3.1]{AIM}. We will provide an argument for the sake of completeness. 
\begin{prop}
\label{non_minimal_gen}
Let $\bsz=z_1,\ldots,z_\ell$ be a sequence of elements having degree $dg$ which generates $\nn^c$ modulo $\mm S$. Denote $\beta^{\bsz}_{i,j}(M)=\dim_l H_i(k\dtensor {R^{(1/d)}} K[\bsz;M])_j$ for each $i,j\in \Z$, and
$$
P^{\vphi,\bsz}_M(t,y)=\sum_i\sum_j \beta^{\bsz}_{i,j}(M)t^iy^j\in \Z[[t,y]]
$$
the corresponding generating function. Denote
\[
\reg_{\vphi,\bsz}M=\sup\left\{\frac{j-idg}{dg}:\beta^{\bsz}_{i,j}(M)\neq 0 \right\}.
\]
Let $n$ be the minimal number of homogeneous generators of $\nn^c/\mm S$. Then for any set $\bsx=x_1,\ldots,x_n$ of minimal generators with degree $dg$ of $\nn^c$ modulo $\mm S$, we have
\[
P^{\vphi,\bsz}_M(t,y)=P^{\vphi,\bsx}_M(t,y)(1+ty^{dg})^{\ell-n},
\]
and $\reg_{\vphi,\bsz}M=\reg_{\vphi,\bsx}M$.
\end{prop}
\begin{proof}
Firstly, let $\bsu=u_1,\ldots,u_r$ be a minimal system of generators having degree $g$ of $\mm$. Let $\bsx=x_1,\ldots,x_n$ be a minimal system of generators having degree $dg$ of $\nn^c$ modulo $\mm S$.

Denote by $\overline{\bsu}$ the residue class of $\bsu \subseteq R^{(1/d)}$ in $k$, then we have the third isomorphism in the following chain
\begin{align*}
k\dtensor{R^{(1/d)}}K[\bsz,\vphi(\bsu);M]
&\simeq k\dtensor{R^{(1/d)}}\left(K[\vphi(\bsu);S]\dtensor{S} K[\bsz;M]\right)\\
&\simeq \left(k \dtensor{R^{(1/d)}}K[\bsu;R^{(1/d)}]\right) \dtensor{R^{(1/d)}} 
K[\bsz;M]\\
&\simeq K[\overline{\bsu};k] \dtensor{R^{(1/d)}} K[\bsz;M] \\
&\simeq 
\left(\bigoplus_{u=0}^r\Sigma^uk(-udg)^{\binom{r}{ 
u}}\right)\dtensor{R^{(1/d)}}K[\bsz;M]\\
&\simeq \bigoplus_{u=0}^r 
\Sigma^u \left(k\dtensor{R^{(1/d)}}K[\bsz;M]\right)(-udg)^{\binom{r}{u}}.
\end{align*}
The fourth isomorphism holds since $K[\overline{\bsu};k]$ has differential $0$. From the above chain, we obtain
\[
P^{\vphi,\bsz,\vphi(\bsu)}_M(t,y)=(1+ty^{dg})^rP^{\vphi,\bsz}_M(t,y).
\]
This equality also implies that
\[
P^{\vphi,\bsx,\vphi(\bsu)}_M(t,y)=(1+ty^{dg})^rP^{\vphi,\bsx}_M(t,y).
\]
Secondly, note that $\nn^c$ is generated by the sequence $\bsz,\vphi(\bsu)$. Similar statement holds for the sequence $\bsx,\vphi(\bsu)$. Choose $\bsv=v_1,\ldots,v_m$ a minimal system of generators with degree $dg$ of $\nn^c$. Let $A$ be the graded Koszul complex on $(\ell+r-m)$ zeros, regarded elements of degree $dg$ of $S$. By standard arguments, we have
\[
K[\bsz,\vphi(\bsu);M]=K[\bsv;M]\otimes_S A.
\]
This yields
\[
P^{\vphi,\bsz,\vphi(\bsu)}_M(t,y)=(1+ty^{dg})^{\ell+r-m}P^{\vphi,\bsv}_M(t,y).
\]
Combining with the similar identity for $P^{\vphi,\bsx,\vphi(\bsu)}_M(t,y)$ and the fact that $\ell\ge n$, we get
\[
P^{\vphi,\bsz,\vphi(\bsu)}_M(t,y)=(1+ty^{dg})^{\ell-n}P^{\vphi,\bsx,\vphi(\bsu)}_M(t,y).
\]
This gives us the desired equality of formal power series
\[
P^{\vphi,\bsz}_M(t,y)=(1+ty^{dg})^{\ell-n}P^{\vphi,\bsx}_M(t,y).
\]
In detail, this means that for all $i,j$,
\[
\beta^{\bsz}_{i,j}(M) = \sum_{u=0}^{\ell-n} \binom{\ell-n}{u}\beta^{\bsx}_{i-u,j-udg}(M).
\]
Therefore
\[
\reg_{\vphi,\bsz}M=\sup\left\{\frac{j-idg}{dg}:\beta^{\bsz}_{i,j}(M)\neq 0 \right\}=\reg_{\vphi,\bsx}M,
\]
as claimed.
\end{proof}
The regularity over a homomorphism is unchanged by tensoring with appropriate Koszul complexes.
\begin{lem}
\label{Koszul_invariance}
For any finite sequence $\bsv$ of elements of degree $dg$ in $S$, we have
\[
\reg_{\vphi}M=\reg_{\vphi}K[\bsv;M].
\]
\end{lem}
\begin{proof}
Let $\bsx$ be a minimal generating set of elements of degree $dg$ for the module $\nn^c/\mm S$. Then
\[
K[\bsx;K[\bsv;M]] \cong K[\bsx,\bsv;M],
\]
hence the conclusion follows by an application of Proposition \ref{non_minimal_gen}.
\end{proof}

The following theorem is the main result in Section \ref{reg_over_morphism}. It has the same spirit as \cite[Theorem 7.2.3]{AIM} but the proof requires a precise bookkeeping of degrees, so we will carry out the details.
\begin{thm}
\label{finite_length}
Let $\bsy$ be a finite sequence of elements of degree $dg$ in $S$. Assume that there exists an integer $r\ge 1$ such that
\[
\nn^r\subseteq \bsy S+\mm S+\Ann_{\dcat{S}}M.
\]
Then there is an equality
\[
\reg_{\vphi}M = \sup\left\{\frac{j-idg}{dg}: H_i(k\dtensor {R^{(1/d)}}K[\bsy;M])_j \neq 0 \right\}.
\]
\end{thm}
\begin{rem}
\label{artinian_case}
If $S/\mm S$ is an artinian ring, then we can choose $\bsy=\emptyset$ in Theorem \ref{finite_length}. In particular, 
\[
\reg_{\vphi}M = \sup\left\{\frac{j-idg}{dg}: H_i(k\dtensor {R^{(1/d)}}M)_j \neq 0 \right\}.
\]
If moreover, $R\to S$ is a finite morphism so that $M$ has degreewise finite 
homology over $R$, then the last equality says that 
$\reg_{\vphi}M=\reg_{R^{(1/d)}} M$.
\end{rem}
We start by proving the following result; it was suggested by Lemma 1.2.3 in \cite{AIM}.
\begin{lem}
\label{triangle}
Let $X(-s)\xrightarrow{\theta}X\to C \to $ be a triangle in $\dgrcat S$ where $s\ge 1$. Assume that the self-composition map $H(\theta)^r:H(X)(-rs)\to H(X)$ is zero for some $r\ge 1$ and $H_i(X)$ has finite length for each $i$. Then for each $i,j$, there are inequalities
\begin{align*}
&\ell (H_i(C)_j)\le \ell(H_i(X)_j)+\ell(H_{i-1}(X)_{j-s}),\\
&\ell (H_i(X)_j)\le \sum_{m=0}^{r-1}\ell(H_{i+1}(C)_{j+(m+1)s}),
\end{align*}
where for a finite length $S$-module $N$, $\ell(N)$ denotes $\dim_l N$.
\end{lem}
\begin{proof}
Denote $\alpha_{i,j}=H_i(\theta)_j:H_i(X)_j \to H_i(X)_{j+s}$. For each $i,j$, we have exact sequences
\[
0 \to \Coker \alpha_{i,j-s} \to H_i(C)_j\to \Ker \alpha_{i-1,j-s} \to 0,
\]
and
\[
0\to \Ker \alpha_{i,j-s} \to H_i(X)_{j-s} \xrightarrow{\alpha_{i,j-s}}H_i(X)_j \to \Coker \alpha_{i,j-s} \to 0.
\]
So there are formulas
\begin{align}
\begin{split}
&\ell(H_i(C)_j)= \ell(\Coker \alpha_{i,j-s}) +\ell(\Ker \alpha_{i-1,j-s}), \label{ineq_H_iC}\\
&\ell (\Coker \alpha_{i,j-s}) \le \ell(H_i(X)_j),\\
&\ell (\Ker \alpha_{i,j-s}) \le \ell(H_i(X)_{j-s}).
\end{split}
\end{align}
From \eqref{ineq_H_iC}, it is clear that
\[
\ell(H_i(C)_j)\le \ell(H_i(X)_j)+ \ell(H_{i-1}(X)_{j-s}).
\]
By hypothesis, $\alpha_{i,j+(r-1)s}\circ \cdots \circ \alpha_{i,j+s}\circ \alpha_{i,j}=0$. Therefore examining the sequence
\[
H_i(X)_j \xrightarrow{\alpha_{i,j}} H_i(X)_{j+s} \to \cdots \to H_i(X)_{j+(r-1)s} \xrightarrow{\alpha_{i,j+(r-1)s}} H_i(X)_{j+rs},
\]
we obtain
\[
\ell(H_i(X)_j) \le \sum_{m=0}^{r-1}\ell(\Ker \alpha_{i,j+ms}).
\]
Combining with \eqref{ineq_H_iC}, we finish the proof of the lemma.
\end{proof}

\begin{proof}[Proof of Theorem \ref{finite_length}]
Denote $L=K[\bsy;M]$. Using Lemma \ref{Koszul_invariance}, we have $\reg_{\vphi}M=\reg_{\vphi}L$. So it is enough to show
\begin{equation}
\label{key_ineq}
\reg_{\vphi}L = \sup\left\{\frac{j-idg}{dg}: H_i(k\dtensor {R^{(1/d)}}L)_j \neq 0 \right\}.
\end{equation}
Let $\bsx=x_1,\ldots,x_n$ be a minimal generating set of degree-$dg$ elements of $\nn^c$ modulo $\mm S$. Denote $L^{(\nu)}=K[x_1,\ldots,x_{\nu};L]$ for all $\nu=0,1,\ldots,n.$ We will show that for all $i\in \Z$ and all $\nu=0,1,\ldots,n$:
\begin{enumerate}
\item $H_i(k\dtensor {R^{(1/d)}}L^{(\nu)})$ is a finite $S$-module,
\item the $S$-module $H_i(k\dtensor {R^{(1/d)}}L^{(\nu)})$ is annihilated by $\nn^r$.
\end{enumerate}
Indeed, (i) is clear since $M\in \dgrcatp S$. Using the hypothesis and Lemma \ref{lem_ann_dtensor}, we get
\[
\nn^r\subseteq \mm S+\bsy S+\Ann_{\dcat{S}}M \subseteq \mm S+\Ann_{\dcat{S}}L.
\]
Now (ii) follows from the inclusions
\[
\nn^r \subseteq \mm S+\Ann_{\dcat{S}}L \subseteq \mm S+\Ann_{\dcat{S}}L^{(\nu)}\subseteq \Ann_{\dcat S}(k\dtensor {R^{(1/d)}}L^{(\nu)})
\]
for each $\nu$. 

For each $\nu$, $L^{(\nu+1)}$ is the mapping cone of the morphism $\lambda^{(\nu)}: L^{(\nu)}(-dg)\to L^{(\nu)}$, where the latter is the multiplication by $x_{\nu+1}$.  Moreover, we have a triangle in $\dgrcat S$
\[
k\dtensor {R^{(1/d)}} L^{(\nu)}(-dg) \xrightarrow{k\dtensor {R^{(1/d)}}\lambda^{(\nu)}} k\dtensor {R^{(1/d)}} L^{(\nu)} \to k\dtensor {R^{(1/d)}} L^{(\nu+1)} \to.
\]
For each $i,j\in \Z$, denote $\beta_{i,j}(L^{(\nu)})=\ell\left(H_i(k\dtensor 
{R^{(1/d)}}L^{(\nu)})_j\right)$, which is a finite number thanks to (i). Let
\[
r(\nu)=\sup\left\{\frac{j-idg}{dg}: \beta_{i,j}(L^{(\nu)}) \neq 0 \right\}.
\]
Because of (ii), we can apply Lemma \ref{triangle} to get the following inequalities
\begin{align}
&\beta_{i,j}(L^{(\nu+1)})\le \beta_{i,j}(L^{(\nu)}) + \beta_{i-1,j-dg}(L^{(\nu)}) \label{ineq_Lnu}.\\
&\beta_{i,j}(L^{(\nu)}) \le \sum_{m=0}^{r-1}\beta_{i+1,j+(m+1)dg}(L^{(\nu+1)}) \label{ineq_Lnu+1}.
\end{align}
Now we will show that for every $\nu$,
\[
r(\nu) = r(\nu+1).
\]
First, take any $i,j$ such that $(j-idg)/dg > r(\nu)$. Then obviously
\[
\frac{(j-dg)-(i-1)dg}{dg} = \frac{j-idg}{dg}> r(\nu).
\]
Therefore $\beta_{i,j}(L^{(\nu)})=\beta_{i-1,j-dg}(L^{(\nu)})=0$. From \eqref{ineq_Lnu}, we also get $\beta_{i,j}(L^{(\nu+1)})=0$. Hence $r(\nu+1) \le r(\nu)$.

Second, take any $i,j$ such that $(j-idg)/dg > r(\nu+1)$. Then for all $m\ge 0$,
\[
\frac{j+(m+1)dg-(i+1)dg}{dg} = \frac{j-idg}{dg}+m\ge \frac{j-idg}{dg} >r(\nu+1).
\]
So $\beta_{i+1,j+(m+1)dg}(L^{(\nu+1)})=0$ for $m=0,\ldots,r-1$. Using \eqref{ineq_Lnu+1}, we infer that $\beta_{i,j}(L^{(\nu)})=0$. Hence $r(\nu)\le r(\nu+1)$. So $r(\nu)=r(\nu+1)$, as claimed.

All in all, we obtain $r(0)=r(1)=\cdots= r(n)$. Finally, note that $r(n)=\reg_{\vphi}L$ and
\[
r(0)=\sup\left\{\frac{j-idg}{dg}: H_i(k\dtensor {R^{(1/d)}}L)_j \neq 0 \right\},
\]
so \eqref{key_ineq} is proved. The proof of the theorem is now completed.
\end{proof}

\section{Factorizations, comparison theorem and composition of homomorphisms}
\label{factorizations_compositions}
\subsection*{Factorizations}
Let $\vphi: (R,\mm,k)\to (S,\nn,l)$ be again a homomorphism of order $d$, $R$ is generated over $k$ in degree $g$ and $S$ is generated over $l$ in degree $h$, where $d,g,h\ge 1$. Instead of working with $\vphi$, we can work with a morphism from a certain polynomial extension of $R$ to $S$.

In detail, let $c\in \N$ be such that $dg=ch$. Let $y_1,\ldots,y_m$ be a 
generating set of $\nn$ modulo $\mm S$ for which $\deg y_i=h$ for all $i$. Let 
$R[t_1,\ldots,t_m]$ be a polynomial extension of $R$, where the $t_i$ are 
variables of degree $g$. Consider the morphism $\vphi':R[t_1,\ldots,t_m] \to S$ 
mapping $t_i$ to $y_i^c$ for each $i$. We have a commutative diagram
\begin{displaymath}
\xymatrix{                     & R[t_1,\ldots,t_m] \ar^{\textstyle \vphi'}[rd]   &\\
           R \ar[ru] \ar^{\textstyle \vphi}[rr]           &                                   &S}
\end{displaymath}
Clearly $\vphi'$ has order $d$ and $S/(\mm S+(t_1,\ldots,t_m)S)$ is artinian. We will call such a pair $(R[t_1,\ldots,t_m],\vphi')$ an {\em artinian factorization} of $\vphi$. Regularity of a complex computed over $\vphi$ or over a factorization are the same because of the following
\begin{thm}[Factorization]
\label{factor_and_reg}
Let $R[t_1,\ldots,t_m]$ be an arbitrary polynomial extension of $R$ (where $m\ge 1$, the variables $t_i$ have degree $g$). Let $\vphi': R[t_1,\ldots,t_m]\to S$ be a homomorphism of order $d$ such that $\vphi$ factors through $\vphi'$. Then for any $M\in \dgrcatp S$, there is an equality
\[
\reg_{\vphi} M =\reg_{\vphi'}M.
\]
\end{thm}
\begin{proof}
Let $\widetilde{\vphi}:R^{(1/d)}\to S$ be the induced homogeneous morphism. From the definition, clearly $\reg_{\vphi}M=\reg_{\widetilde{\vphi}}M$. Moreover, $(R[t_1,\ldots,t_m])^{(1/d)}\cong R^{(1/d)}[t'_1,\ldots,t'_m]$, a polynomial extension of $R^{(1/d)}$, where the new variables $t'_i$ have degree $dg$. Therefore we can assume from the beginning that $\vphi$ is a homogeneous morphism, namely $d=1$. So $g=hc$.

Let $F$ be the minimal graded free resolution of $k$ over $R$. Denote $Q=R[t_1,\ldots,t_m]$ and $\bst=t_1,\ldots,t_m$. Clearly $K[\bst;Q]\simeq R$ over $Q$. Since $R\to Q$ is flat, we have the following isomorphisms in $\dgrcat Q$
\[
F\dtensor R K[\bst;Q] \simeq k\dtensor R K[\bst;Q] \simeq k\dtensor R R \simeq k.
\]
Let $\mm_Q$ be the graded maximal ideal of $Q$. Since $\vphi$ factors through $\vphi'$, we have $\mm S \subseteq \mm_Q S$. Let $\bsx$ be a finite sequence of elements of degree $g$ in $S$ which minimally generates $\nn^c$ modulo $\mm S$. Then we have isomorphisms in $\dgrcat Q$:
\begin{align*}
k\dtensor {Q} K[\bsx;M] &\simeq (F\dtensor R K[\bst;Q]) \dtensor Q K[\bsx;M]\\
                        &\simeq F\dtensor R (K[\bst;Q] \dtensor Q K[\bsx;M])\\
                        &\simeq F\dtensor R K[\vphi'(\bst),\bsx;M]\\
                        &\simeq k\dtensor R K[\bsx; K[\vphi'(\bst);M]].
\end{align*}
Note that $\bsx$ also generates $\nn^c$ modulo $\mm_QS$, therefore, Proposition 
\ref{non_minimal_gen} and Lemma \ref{Koszul_invariance} imply the two 
equalities at the two ends in the following strand
$$\reg_{\vphi'}M=\reg_{\vphi',\bsx}M = \reg_{\vphi}K[\vphi'(\bst);M]=\reg_{\vphi}M.$$
The middle equation follows from the isomorphisms above. The proof of the theorem is now complete.
\end{proof}

\subsection*{Comparison Theorem}
The first part of the next result shows the stability of regularity over a 
homomorphism with respect to taking fractional Veronese. The second part gives a 
connection between the fractional Veronese and the Veronese grading, namely 
provides a comparison between the regularity of $M$ over $R^{(1/d)}\to S$ 
and that over the induced morphism $R\to S^{(d)}$.
\begin{thm}
\label{reg_fract_Veronese}
The following statements hold true for any complex $M\in \dgrcatp S$:
\begin{enumerate}[\quad \rm(i)]
\item For any $s\ge 1$, denote by $\vphi^{(1/s)}$ the induced homomorphism $R^{(1/s)}\to S^{(1/s)}$. Then we have
\[
\reg_{\vphi}M=\reg_{\vphi^{(1/s)}}M^{(1/s)}.
\]
\item \textnormal{(Comparison)} Denote by $\widehat{\vphi}$ the induced 
homogeneous morphism $R\to S^{(d)}$. Then
\[
\reg_{\widehat{\vphi}} M \le \reg_{\vphi}M \le 
\reg_{\widehat{\vphi}} M+\frac{d-1}{dg}.
\]
\end{enumerate}
\end{thm}
\begin{proof}
Let $\bsx=x_1,\ldots,x_n$ be minimal homogeneous generators of $\nn^c$ modulo 
$\mm S$.

For (i): let $\bsx'$ be the sequence $x_1,\ldots,x_n$ regarded as elements of $S^{(1/s)}$. It is not hard to see that in $\dgrcat {S^{(1/s)}}$,
\[
K[\bsx';M^{(1/s)}]=K[\bsx;M]^{(1/s)}.
\]
Denote $t^{\vphi}_i(M)=\sup\{j:\beta^{\vphi}_{i,j}(M)\neq 0\}$ for each $i\in \Z$. Using the exactness of the functor $(\cdot)^{(1/s)}$, we have
\begin{align*}
k\dtensor {R^{(1/ds)}} K[\bsx';M^{(1/s)}] &\simeq k^{(1/s)}\dtensor {R^{(1/ds)}} K[\bsx;M]^{(1/s)}\\
                                          &\simeq (k\dtensor {R^{(1/d)}} K[\bsx;M])^{(1/s)}.
\end{align*}
Therefore $t^{\vphi^{(1/s)}}_i(M^{(1/s)})=s\cdot t^{\vphi}_i(M)$ for all $i$. The desired equality follows immediately.

For (ii): choose an artinian factorization $(R[t_1,\ldots,t_m],\vphi')$ of 
$\vphi$. By Theorem \ref{factor_and_reg}, we get $\reg_{\vphi}M=\reg_{\vphi'}M$ 
and $\reg_{\widehat{\vphi}} M=\reg_{\widehat{\vphi'}} M$. Therefore we 
can replace $R$ by $R[t_1,\ldots,t_m]$ and assume that $S/\mm S$ is artinian.

Applying Remark \ref{artinian_case}, we see that 
\[
\reg_{\vphi}M=\sup\left\{\frac{j-idg}{dg}:H_i(k\dtensor{R^{(1/d)}}M)_j\neq 0\right\},
\] 
and, as $\widehat{\vphi}:R\to S^{(d)}$ is a map of degree $0$, that 
\[
\reg_{\widehat{\vphi}}M=\sup\left\{\frac{j-ig}{g}:H_i(k\dtensor{R} M)_j\neq 
0\right\}.
\]
Denote 
\begin{align*}
& t_i=\sup\{j: H_i(k\dtensor{R^{(1/d)}}M)_j\neq 0\},\\
& s_i=\sup\{j: H_i(k\dtensor{R} M)_j\neq 0\},
\end{align*}
 for each $i$. Then $\reg_{\vphi}M=\sup_{i\in \Z}\{(t_i-idg)/dg\}$. Let $G$ be the minimal graded free resolution of $k$ over $R^{(1/d)}$. Then by Lemma \ref{functor_varPhi}, $\Phi(G)$ is the minimal graded free resolution of $k$ over $R$. Hence 
$$
\Phi(k\dtensor{R^{(1/d)}}M)=\Phi(G\otimes_{R^{(1/d)}}M) \simeq \Phi(G)\otimes_R 
M \simeq k\dtensor{R} M
$$ 
as $S^{(d)}$-complexes. Therefore $H_i(k\dtensor{R} M)_j \neq 0$ for some $j$ if 
and only if $H_i(k\dtensor{R^{(1/d)}}M)_{dj+r}\neq 0$ for some $0\le r\le d-1$. 
In particular, we get
\[
s_i\le \frac{t_i}{d} \le s_i+\frac{d-1}{d},
\]
for each $i$. Combining with
\[
\reg_{\widehat{\vphi}} M = \sup_{i\in \Z}\left\{\frac{s_i-ig}{g}\right\},
\]
the conclusion follows.
\end{proof}

\subsection*{Composition of homomorphisms}
First we consider the behavior of regularity over a homomorphism along 
polynomial extensions. Let $S[t]$ be a polynomial extension of $S$ where $\deg 
t=h$. Let $R[t']$ be a polynomial extension of $R$ where $\deg t'=g$. Let 
$\vphi[t]$ be the morphism $R[t']\to S[t]$ which restricts to $\vphi$ on $R$ 
and $\vphi[t](t')=t^c$. Clearly $\vphi[t]$ also has order $d$. Denote 
$M[t]=M\otimes_S S[t]$ for each $M\in \dgrcatp S$. We have
\begin{prop}
\label{polynomial_ext}
For any $M\in \dgrcatp S$, there is an equality
\[
\reg_{\vphi[t]}M[t]=\reg_{\vphi}M + (1-\frac{h}{dg}).
\]
\end{prop}
\begin{proof}
First, consider the case $S/\mm S$ is artinian. Then $S[t]/(\mm +(t'))S[t]$ is also artinian. We know from Remark \ref{artinian_case} that 
\begin{align*}
\reg_{\vphi}M&=\sup\left\{\frac{j-idg}{dg}:H_i(k\dtensor{R^{(1/d)}}M)_j\neq 0\right\},\\
\reg_{\vphi[t]}M[t]&=\sup\left\{\frac{j-idg}{dg}:H_i(k\dtensor{R^{(1/d)}[z]}M[t])_j\neq 0\right\},
\end{align*}
where $R^{(1/d)}[z] \cong (R[t'])^{(1/d)}$ is a polynomial extension of $R^{(1/d)}$ with $\deg z=dg$. Denote $A=R^{(1/d)}$. 

Since $z$ maps to $t^c\in S[t]$, we have isomorphisms of graded free $S[z]$-modules $S[t]\cong  \bigoplus_{j=0}^{c-1}S[z]t^j \cong \bigoplus_{j=0}^{c-1}S[z](-hj)$. Therefore $M[t]\simeq \bigoplus_{j=0}^{c-1}(M[z])(-hj)$ as $S[z]$-complexes. We obtain that $k\dtensor{A[z]}M[t] \cong \bigoplus_{j=0}^{c-1}(k\dtensor{A[z]}M[z])(-hj).$
Let $F$ be the minimal graded $S$-free resolution of $k\dtensor{A}M$. Then because of the flatness of $S\to S[z]$, $F\otimes_S S[z]$ is the minimal graded $S[z]$-free resolution of $k\dtensor{A[z]}M[z]$. In particular
\[
\sup\{j:H_i(k\dtensor{A}M)_j\neq 0\}=\sup\{j:H_i(k\dtensor{A[z]}M[z])_j\neq 0\}.
\] 
Hence $\reg_{\vphi[t]}M[t]=\reg_{\vphi} M + (h(c-1)/dg)$, as desired.

Now consider the general case. Let $(R[t'_1,\ldots,t'_m],\vphi')$ be an artinian factorization of $\vphi$. Then from Theorem \ref{factor_and_reg}, $\reg_{\vphi}M=\reg_{\vphi'}M$. Denote by $(\vphi[t])'$ the induced homomorphism $R[t'_1,\ldots,t'_m,t']\to S[t]$. The two morphisms $\vphi'[t]$ and $(\vphi[t])'$ are equal, so we have $\reg_{\vphi'[t]}M[t]=\reg_{(\vphi[t])'}M[t]$. Now applying the previous case for the morphism $\vphi'$, we obtain the desired conclusion.
\end{proof}

The next result is the key to constructing certain complexes in the proof 
of Theorem \ref{orderge2}.
\begin{thm}[Composition of homomorphisms]
\label{composition}
Let $(R',\mm',k')\xrightarrow{\pi'} (R,\mm,k) \xrightarrow{\pi}(S,\nn,l)$ be 
homomorphisms of graded algebras such that $\order \pi'=d'$ and $\order \pi= d$. 
Assume that $R',R,S$ are generated as an algebra in degrees $g',g,h\ge 1$, 
respectively. 
Let $L\in \dgrcatp R$ and $N\in \dgrcatp S$ be such that $L, N\not\simeq 0$, $\reg_{\pi'}L<\infty$ and $\reg_{\pi}N <\infty$. Denote by $P$ the complex of graded $S$-modules $L^{(1/d)}\dtensor {R^{(1/d)}}N$. Then we also have
\[
\reg_{\pi \circ \pi'} P <\infty.
\]
\end{thm}
The following diagram illustrates the statement of Theorem \ref{composition}.
\begin{displaymath}
\xymatrix{                 &   &       &                                           & P=L^{(1/d)}\dtensor {R^{(1/d)}}N \ar@{.}[d]\\
R' \ar[rrr]^{\textstyle\pi'} \ar@/^1.9pc/[rrrr]^{\textstyle \pi\circ \pi'}&  & & R \ar[r]^{\textstyle\pi} \ar@{.}[d]  & S \ar@{.}[d] \\
& & & L & N}
\end{displaymath}
The dotted line between $R$ and $L$ simply means that $L$ is a complex of 
$R$-modules, and so on.
\begin{proof}
Replacing $\pi',\pi$ by suitable artinian factorizations and polynomial extensions, we can assume that $R/\mm'R$ and $S/\mm S$ are artinian rings.

Indeed, consider the diagram below where
\begin{enumerate}
\item $(R[t_1,\ldots,t_m],\tau)$ is an artinian factorization of $\pi$,
\item $R'[z_1,\ldots,z_m]$ is a polynomial extension of $R'$ such that degree of each new variable is $g'$,
\item the morphism $\pi'$ is extended to a homomorphism $\pi''=\pi'[t_1,\ldots,t_m]:R'[z_1,\ldots,z_m] \to R[t_1,\ldots,t_m]$ mapping $z_i$ to $t_i^c$, where $c$ is the unique integer such that $d'g'=gc$ (clearly $\order \pi'' =d'$), and, 
\item $(R'[y_1,\ldots,y_n,z_1,\ldots,z_m],\tau')$ is an artinian factorization of $\pi''$.
\end{enumerate}
\begin{displaymath}
\xymatrix{R'[y_1,\ldots,y_n,z_1,\ldots,z_m] \ar[rrd]^{\textstyle \tau'}     &  &                       &                        & &\\  R'[z_1,\ldots,z_m] \ar[rr]^{\textstyle \pi''}\ar@.[u]             &  &R[t_1,\ldots,t_m] \ar[rrrd]^{\textstyle \tau} &             & &\\
          R' \ar[rr]^{\textstyle \pi'} \ar@.[u]         &  & R \ar[rrr]^{\textstyle \pi} \ar@.[u]          &            & &S}
\end{displaymath}
Let $L''=L\otimes_R R[t_1,\ldots,t_m]$. We will show that the statement of our result is not affected if we replace $R'$ by $R'[y_1,\ldots,y_n,z_1,\ldots,z_m]$, $R$ by $R[t_1,\ldots,t_m]$ and $L$ by $L''$.

Firstly, from Theorem \ref{factor_and_reg} and Proposition \ref{polynomial_ext},
\[
\reg_{\tau'}L''=\reg_{\pi''}L''=\reg_{\pi'}L+m(1-g/d'g')<\infty
\]
and
\[
\reg_{\pi}N=\reg_{\tau}N <\infty.
\]

We also have $L''^{(1/d)}\simeq L^{(1/d)}\dtensor {R^{(1/d)}} R^{(1/d)}[t'_1,\ldots,t'_m]$; the latter ring is a polynomial extension of $R^{(1/d)}$ with variables of degree $dg$. Denote by $P'$ the $S$-complex $L''^{(1/d)}\dtensor {R[t_1,\ldots,t_m]^{(1/d)}} N$, we have isomorphisms in $\dgrcat {S}$:
\begin{align*}
P'              &\simeq  (L^{(1/d)}\dtensor {R^{(1/d)}} R^{(1/d)}[t'_1,\ldots,t'_m])\dtensor {R^{(1/d)}[t'_1,\ldots,t'_m]} N \nonumber \\
                &\simeq  L^{(1/d)}\dtensor {R^{(1/d)}} N \simeq  P.
\end{align*}
But $\tau \circ \tau'$ is an artinian factorization of $\pi \circ \pi'$, hence $\reg_{\pi \circ \pi'}P=\reg_{\tau\circ\tau'}P'$.  Therefore we can make the artinian assumptions from above.

Furthermore, replacing $R$ by $R^{(1/d)}$, $R'$ by $R'^{(1/dd')}$, $N$ by $N^{(1/d')}$ and using Theorem \ref{reg_fract_Veronese}, we can assume that $d=d'=1$. Now $P=L\dtensor{R}N$ and we have to show that $\reg_{\pi \circ \pi'}P<\infty$.

From the associativity of derived tensor product
\[
(k'\dtensor {R'}L)\dtensor R N \simeq k'\dtensor {R'} (L\dtensor R N),
\]
one obtains the standard spectral sequence
\[
\Tor^R_{\ell}(\Tor^{R'}_{\ell'}(k',L),N) \Rightarrow \Tor^{R'}_{\ell+\ell'}(k',P).
\]
 Since $R/\mm' R$ is artinian, $\Tor^{R'}_{\ell}(k',L)$ has finite length as an $R$-module for any $\ell$. Consequently we have for all $i,j\in \Z$:
\begin{align*}
\dim_l \Tor^{R'}_i(k',P)_j &\le \sum_{\ell+\ell'=i}\dim_l  \Tor^R_\ell(\Tor^{R'}_{\ell'}(k',L),N)_j \\
                           &\le  \sum_{\ell+\ell'=i}\sum_u \dim_k \Tor^{R'}_{\ell'}(k',L)_u \dim_l \Tor^R_\ell(k(-u),N)_j\\
                           &=\sum_{\ell+\ell'=i}\sum_u \dim_k \Tor^{R'}_{\ell'}(k',L)_u \dim_l \Tor^R_\ell(k,N)_{j-u}.
\end{align*}
The second inequality follows by filtering $\Tor^{R'}_{\ell'}(k',L)$ appropriately.

Define $t^{\pi'}_i(L)=\sup\{j:H_i(k\dtensor{R'}L)_j\neq 0\}$ and similarly $t^{\pi}_i(N), t^{\pi'\circ \pi}_i(P)$. From the last inequality, we get for all $i$,
\begin{align*}
t^{\pi'\circ \pi}_i(P) &\le \sup_{\ell\ge \inf N}\{t^{\pi'}_{i-\ell}(L)+t^{\pi}_\ell(N)\}\\
                       &\le \sup_{\ell\ge \inf N}\{g'(i-\ell)+g'\reg_{\pi'}L+ g\ell+g\reg_{\pi} N\}\\                              
                       &\le g'i +g'\reg_{\pi'}L+ g\reg_{\pi} N +(g-g')\inf N.
\end{align*}
The last inequality follows since $g'\ge g$, which in turn holds because $R'\to R$ is degree-preserving. Finally, from Remark \ref{artinian_case} and the last string we conclude that
\[
\reg_{\pi'\circ \pi} P \le \reg_{\pi'}L + \frac{g\reg_{\pi} N +(g-g')\inf N}{g'} <\infty.
\]
\end{proof}
\section{Proofs of the main results}
\label{F_application}
\subsection*{Proof of Theorem \ref{relative_thm}}
Recall that $R,S$ are standard graded algebras over $k,l$ respectively. Denote by $\mm,\nn$ the corresponding graded maximal ideals. The main work in the proof of Theorem \ref{relative_thm} is done via the following statement, which was suggested to the authors by Srikanth Iyengar.

\begin{prop}
\label{reg_of_complex}
For any complex $G\in \dgrcatp R$, we have
\[
\reg_R G \le \sup_{i\in \Z}\{\reg_R H_i(G)-i\}.
\]
\end{prop}
\begin{proof}
Since $H_i(G)=0$ for $i\ll 0$, the minimal graded free resolution of $G$ can be chosen to be a bounded below complex $F$ with $F_i=0$ for $i <\inf G$. Observe that for any $m\in \Z$, we have
\begin{enumerate}
\item $\shift^{-m} F$ is the minimal free resolution of $\shift^{-m}G$, hence,
\[
\reg_R \shift^{-m}G=\sup_{i\in \Z}\{\reg_R \shift^{-m}F_i-i\}= \reg_R G +m,
\]
\item $\reg_R H_i(\shift^{-m}G)-i=\reg_R H_{i+m}(G)-(i+m)+m$, so
\[
\sup_{i\in \Z}\{\reg_R H_i(\shift^{-m}G)-i\} = \sup_{i\in \Z}\{\reg_R H_i(G)-i\} +m.
\]
\end{enumerate}
Therefore by replacing $G$ by $\shift^{-m}G$, both sides of the inequality in question increase by $m$. Hence we can assume that $\inf G=0$.

Denote $B_i=B_i(F)=\img (F_{i+1}\to F_i), Z_i=Z_i(F)=\Ker (F_i\to F_{i-1})$ and $H_i=H_i(F)=H_i(G)$. Since $F$ is minimal, for each $i\ge 0$, $0 \to B_i \to F_i \to F_i/B_i \to 0$ is the beginning of the minimal graded free resolution of $F_i/B_i$. Therefore
\begin{equation}
\label{regF_i}
\reg_R F_i \le \reg_R F_i/B_i,
\end{equation}
and
\[
\reg_R B_i \le \reg_R F_i/B_i+1.
\]
On the other hand, we have the following exact sequence
\begin{equation}
\label{sequence}
0 \to H_i \to F_i/B_i \to B_{i-1} \to 0.
\end{equation}
Hence
\[
\reg_R F_i/B_i \le \max \{\reg_R H_i, \reg_R B_{i-1}\}.
\]
Combining the last two inequalities, one has
\[
\reg_R B_i \le \max\{\reg_R H_i+1, \reg_R B_{i-1}+1\}.
\]
By induction on $i\ge 0$, it is easy to see that
\begin{equation}
\label{regB_iandH_i}
\reg_R B_i \le \max\{\reg_R H_i+1, \reg_R H_{i-1}+2,\ldots, \reg_R H_0+i+1\}.
\end{equation}
Now using \eqref{regF_i}, the sequence \eqref{sequence}, and \eqref{regB_iandH_i}, we have for every $i\ge 0$,
\begin{align*}
\reg_R F_i & \le \reg_R F_i/B_i \\
           & \le \max \{\reg_R H_i, \reg_R B_{i-1}\} \\
           & \le \max \{\reg_R H_i,\reg_R H_{i-1}+1,\ldots, \reg_R H_0+i\}.
\end{align*}
This implies that
\[
\reg_R F_i -i \le \max_{0\le j \le i} \{\reg_R H_j -j \}.
\]
The proof of the result is now complete.
\end{proof}
\begin{rem}
The strict inequality in Proposition \ref{reg_of_complex} may happen. For example, take $S=k[x]/(x^2)$. Consider the Koszul complex on $x$ of $S$:
\[
G: 0\to S(-1)\xrightarrow{\cdot x}S \to 0.
\]
$G$ is the minimal resolution of $G$ itself, so clearly $\reg_R G=0$. On the other hand, $H_1(G)\cong (xS)(-1),H_0(G)=k$,
hence $\sup_{i\in \Z}\{\reg_S H_i(G)-i\}=\reg_S H_1(G)-1=1$.
\end{rem}
At this point, we are ready to prove Theorem \ref{relative_thm}.
\begin{proof}[Proof of Theorem \ref{relative_thm}]
Let $\bsy$ be minimal generators of $\nn$ modulo $\mm S$ and denote $N=K[\bsy;M]$. We have an equality between Poincar\'e series; see \cite[Lemma 1.5.3]{AF}:
\[
P^S_{k\dtensor R N}(t,y)=P^R_k(t,y)P^S_N(t,y).
\]
Denote $t_i=\sup\{j:H_i(k\dtensor R N)_j\neq 0\}$. By Lemma \ref{lem_ann_dtensor}, 
\[
\nn =\mm S+\bsy S\subseteq \Ann_{\dcat S}(k\dtensor R N),
\]
thus $H_i(k\dtensor R N)$ is an $S/\nn$-module for every $i$. But $S$ is a Koszul algebra, hence $\reg_S H_i(k\dtensor R N)=t_i$. Together with Proposition \ref{reg_of_complex}, we have
\begin{align*}
\reg_S k\dtensor R N &\le \sup_{i\in \Z} \{\reg_S H_i(k\dtensor R N)-i \}\\
                     &= \sup_{i\in \Z} \{t_i-i \}=\reg_{\vphi} M.
\end{align*}
The last equality follows from the definition of regularity over a homomorphism. Combining with the equality of Poincar\'e series, the last inequality implies that $\reg_S N + \reg_R k \le \reg_{\vphi} M$. Since $N\in \dgrcatp S$ and $N\not\simeq 0$, we infer that $\reg_S N >-\infty$, and hence $\reg_R k<\infty$. By Avramov-Peeva's theorem, $R$ is Koszul.
\end{proof}
The proof of Theorem \ref{relative_thm} also has the following corollary, which 
extends and gives a new proof of the main result of \cite{AE}.
\begin{cor}
Let $R\to S$ be a finite homomorphism of Koszul algebras over $k$. Let $M$ be a finitely generated graded $S$-module. Then $\reg_S M \le \reg_R M.$
\end{cor}
\begin{proof}
Replacing $R$ by a suitable polynomial extension and using Theorem \ref{factor_and_reg}, we can assume that $R\to S$ is surjective. It is also harmless to assume that $\reg_R M$ is finite (which is actually the case since $R$ is Koszul, but we will not recourse to this fact). Hence from the proof of Theorem \ref{relative_thm} where $\bsy$ is now $\emptyset$, we have
\[
\reg_S M = \reg_S M + \reg_R k \le \reg_R M,
\]
as claimed.
\end{proof}
\begin{rem}
\label{concluding_remark}
There are several applications of Theorem \ref{relative_thm}. If $R\to S$ is a 
finite morphism of standard graded $k$-algebras such that $\reg_R S<\infty$, and 
$S$ is a Koszul algebra then $R$ is also a Koszul algebra. A similar observation 
was made in \cite[Theorem~3.2]{CDR}. In particular, in the case $S=k$ and $R\to 
k$ is the canonical surjection, Theorem \ref{relative_thm} gives the 
Avramov-Peeva's characterization of Koszul algebras. (But one does not obtain a 
new proof of this last result.)
\end{rem}

\subsection*{Proofs of Theorems \ref{orderge2} and \ref{main1}} Now we will prove Theorem \ref{orderge2}. This theorem combines with Remark \ref{artinian_case} and Theorem \ref{reg_fract_Veronese}(ii) to give Theorem \ref{main1}, by letting $\psi=\phi^e$.
The construction via derived tensor in the following argument grew out of a beautiful idea in the proof of \cite[Theorem~5.1]{AHIY}.
\begin{proof}[Proof of Theorem \ref{orderge2}]
Let $d$ be the order of $\psi$. Define inductively a complex $M^i\in \dgrcatp R$ 
as follows: $M^1=M$, and for $i\ge 1$
\[
M^{i+1}=(M^i)^{(1/d)}\dtensor {R^{(1/d)}}M.
\]
For $i\ge 1$, the action of $R$ on $M^{i+1}$ is defined as follows. Consider the diagram $R\xrightarrow{\psi^i} R\xrightarrow{\psi} R$. In the derived tensor product, we view $M$ as a complex over the ring $R$ on the right and $M^i$ as a complex over the middle ring of the diagram. The action of $R$ on $M^{i+1}=(M^i)^{(1/d)}\dtensor {R^{(1/d)}}M$ is induced by the action of $R$ on the second variable.

It is easy to see that $M^i \not\simeq 0$ for every $i\ge 1$. We will show by induction that for every $i\ge 1$, $\reg_{\psi^i}M^i<\infty$. For $i=1$, this is known from above. For $i\ge 2$, it is enough to apply Theorem \ref{composition} for the homomorphisms $R\xrightarrow{\psi^{i-1}} R\xrightarrow{\psi} R$ and complexes $L=M^{i-1}$ and $N=M$.

Denote by $\reg R$ the absolute Castelnuovo-Mumford regularity of $R$. Choose $d^i\ge (\reg R+1)/2$, then $R^{(d^i)}$ is Koszul; see \cite[Theorem~2]{ERT} and \cite[Theorem~3.6 and Remark 3.7]{CDR}. From Theorem \ref{reg_fract_Veronese}(ii) we have
\[
\reg_{\widehat{\psi^i}} M^i\le \reg_{\psi^i}M^i<\infty.
\]
Applying Theorem \ref{relative_thm} for the morphism 
$R\xrightarrow{\widehat{\psi^i}}R^{(d^i)}$ and the complex $M^i$, we conclude 
that $R$ is Koszul.
\end{proof}

\begin{rem}
\label{char_0}
The following example gives a class of graded algebras over a field $k$ of characteristic $0$ which possess endomorphisms of order at least $2$. Let $\Lambda$ be a positive affine monoid, namely, $\Lambda$ is a finitely generated submonoid of $\N^m$ for some $m\ge 1$. Let $R=k[\Lambda]$ be the affine monoid ring of $\Lambda$ with the inherited $\N^m$-grading. Let $a_1,\ldots,a_n$ be the minimal generating set of $\Lambda$, and denote by $t^{a_1},\ldots,t^{a_n}$ the corresponding generators of $R$. For any $d\ge 2$, the endomorphism of $R$ mapping $k$ to itself and $t^{a_i}$ to $t^{da_i}$ for all $i=1,\ldots,n$ is an endomorphism of order $d$.
\end{rem}

As a consequence of Theorem \ref{main1}, we have:
\begin{cor}
\label{F-Koszul}
Let $\chara k=p>0$, $k$ be $F$-finite and $R$ be a standard graded $k$-algebra. If for some $e>0$, $\reg_R \up{\phi^e}R <\infty$, then $R$ is a Koszul algebra.
\end{cor}
This result marked a starting point for the research in this paper. We came to this result via an entirely different approach; because of the simplicity of this approach, we record it below. The techniques involved were developed in \cite[Section 5]{BCR}.

\begin{proof}[Alternative proof of Corollary \ref{F-Koszul}]
Denote $V_i=V_i(q,R)$ for $0\le i\le q-1$. By the hypothesis, $r=\reg_R 
\up{\phi^e}R=\max \{\reg_R V_i: i \in \{0,1,\ldots,q-1\} \} <\infty$. Let $M$ be 
a finitely generated graded $R$-module. Denote $M^{(q)}=V_0(q,M)$. Firstly, we 
will show that
\begin{equation}
\label{inequality}
\reg_R M^{(q)} \le \left\lceil \frac{\reg_R M}{q} \right\rceil + r.
\end{equation}
Of course, it suffices to consider the case $\reg_R M<\infty$. Let $G_{\pnt}$ be the minimal graded $R$-free resolution of $M$. Then $\reg_R G_i \le i+\ell$ for all $i \ge 0$, where $\ell=\reg_R M$.

From the exactness of the complex $G_{\pnt}^{(q)} \to M^{(q)} \to 0$ and Lemma \ref{reg_sequence}(ii), we have
\begin{equation}
\label{bounding_reg}
\reg_R M^{(q)} \le \sup \{\reg_R G_i^{(q)}-i:i\ge 0\}.
\end{equation}
Assume that $G_i=\oplus R(-j)^{\beta_{ij}}$. Observe that $R(-j)^{(q)}=V_{u_j}(-\lceil j/q \rceil)$ where $u_j=q\lceil j/q\rceil -j$. Thus we get
\[
G_i^{(q)} = \bigoplus_j \left(V_{u_j}(-\lceil j/q \rceil)\right)^{\beta_{ij}}.
\]
Hence it is clear that
\[
\reg_R G_i^{(q)} \le \left \lceil \frac{\reg_R G_i}{q} \right\rceil + r  \le \left\lceil \frac{i+\ell}{q} \right\rceil +r.
\]
Together with \eqref{bounding_reg}, this implies that
\begin{align*}
\reg_R M^{(q)} & \le   \sup \left \{\left\lceil \frac{i+\ell}{q} \right\rceil +r-i:i\ge 0 \right \} \\
               & \le  \left \lceil \frac{\ell}{q} \right\rceil + r = \left\lceil \frac{\reg_R M}{q} \right\rceil + r,
\end{align*}
proving \eqref{inequality}.

Note that $(R^{(q^i)})^{(q)}=R^{(q^{i+1})}$. Therefore using the inequality \eqref{inequality} and induction we have $\reg_R R^{(q^i)}<\infty$ for all $i >0$. Hence choosing $q^i\ge (\reg R+1)/2$ so that $R^{(q^i)}$ is Koszul, and applying Theorem \ref{relative_thm}, we conclude that $R$ is also Koszul.
\end{proof}

\subsection*{Examples}
We close this section with some examples to illustrate Theorem \ref{main1}.
\begin{ex}
Let $k=\Z/(2)$, $R$ be defined by monomial relations. In the following, the notation $V_i$ stands for $V_i(2,R)$ for $i=0,1$.

(i) If $R=k[x,y]/(xy)$, then $V_0= R, V_1= R/(y) \oplus R/(x)$ where the generators of the second module are $\bar{x}, \bar{y}$, respectively. Therefore $\reg_R V_0=\reg_R V_1=0$. The ring $R$ is Koszul.

(ii) More generally, consider $R=Q/I_G=k[\Delta]$, where $G$ is a finite simple graph on $n$ vertices, $Q=k[x_1,\ldots,x_n]$ and $I_G=(x_ix_j:\{i,j\} ~\textnormal{is an edge of $G$})$. Let $\Delta$ be the associated simplicial complex, i.e. a subset $C$ of $[n]$ belongs to $\Delta$ if and only if $\prod_{i\in C}x_i \notin I_G$. Then for $i=0,1$,
\[
V_i= \bigoplus_{0\le t \le \frac{n-i}{2}}\mathop{\bigoplus_{C \in \Delta}}_{|C|=2t +i}[R/(x_j:j\in N(C))](-t),
\]
where $N(C)$ denotes the sets of vertices of $G$ which are neighbors in $G$ of some element in $C$. By convention $N(\emptyset)=\emptyset$. We prove the equality for $i=0$; similar arguments apply to the remaining case. Observe that $V_0$ is the direct sum of the cyclic modules $R\prod_{i\in C}x_i$, where $C\in \Delta$ and $|C|$ is even. Now for each such face $C$, let $t$ be the degree of $\prod_{i\in C}x_i$ in $V_0$, we have an isomorphism $R\prod_{i\in C}x_i \cong [R/(x_j:j\in N(C))](-t)$. The formula for $V_0$ is now proved.

It is well-known that $R$ is Koszul and for every subset $I$ of $[n]$, it holds that $\reg_R R/(x_i:i\in I)=0$. Hence for $i=0,1$, we have
\[
\reg_R V_i= \max\{t: \textnormal{there exists} ~ C \in \Delta ~\textnormal{such that}~ |C|=2t +i\}.
\]

(iii) If $R=k[x]/(x^3)$, then $V_0=k[x]/(x^2), V_1=k$. The minimal graded free resolution of $V_1=k$ over $R$ is
\[
T_{\pnt}: \cdots \xrightarrow{\cdot x} R(-3) \xrightarrow{\cdot x^2} R(-1) \xrightarrow{\cdot x} \ R.
\]
Concretely $T_{2i}=R(-3i), T_{2i+1}=R(-(3i+1))$ for all $i\ge 0$. In particular, $\reg_R k=\infty$ and $R$ is not Koszul.
\end{ex}
\section{Remarks on the local case and open questions}
\subsection*{Linearity defect}
We can try to extend our main result Theorem \ref{main1} to the local situation. First, let us recall some notations. Let $(R,\mm,k)$ be a Noetherian local ring with maximal ideal $\mm$ and residue field $R/\mm \cong k$. For each finitely generated $R$-module $M$, let $F$ be the minimal free resolution of $M$. Then $F$ has the filtration $\{\Fc^iF\}_{i\ge 0}$ where $\Fc^iF$ is the following complex
\[
\cdots \to F_{i+1} \to F_i \to \mm F_{i-1} \to \cdots \to \mm^iF_0 \to 0.
\]
The {\em linear part} of $F$, denoted by $\linp^R F$ is the associated graded complex of the filtration $\{\Fc^iF\}_{i\ge 0}$. Concretely, $(\linp^R F)_i=\gr_{\mm}(F)(-i)$ for each $i\ge 0$. The {\em linearity defect} $\lind_R M$ of $M$, introduced by Herzog and Iyengar \cite{HIy}, is defined by
\[
\lind_R M =\sup\{i: H_i(\linp^R F)\neq 0\}.
\]
The ring $R$ is called Koszul if $\lind_R k=0$. In contrast to the graded case, it is an open question whether for a local ring $(R,\mm,k)$, $\lind_R k <\infty$ implies that $R$ is a Koszul ring; see \cite[Question 1.14]{HIy}. Our next result says that $\lind_R \up{\phi^e}R<\infty$ for some large $e$ implies $\lind_R k<\infty$.

First we introduce an invariant which is modelled after the number $\nu(R)$ of Takahashi and Yoshino in \cite[Section 3]{TY}.
\begin{defn}
For each maximal sequence of $R$-regular elements ${\bf y}=y_1,\ldots,y_r \in 
\mm^2$, where $r=\depth R$, note that $H^0_{\mm}(R/({\bf y})) \cap \mm^s(R/({\bf 
y}))=0$ for $s\gg 0$. Denote by $\overline{\nu}(R)$ the smallest number $s$ such 
that $H^0_{\mm}(R/({\bf y})) \cap \mm^s(R/({\bf y}))=0$ for some maximal 
sequence of $R$-regular elements ${\bf y}=y_1,\ldots,y_r \in \mm^2$.
\end{defn}
Note that $\nu(R) \le \overline{\nu}(R) <\infty$. The following result is suggested by \cite[Corollary 3.3]{TY} due to Takahashi and Yoshino.
\begin{thm}
\label{F-Koszul-local}
Let $(R,\mm,k)$ be an $F$-finite local ring of characteristic $p>0$. If $\lind_R \up{\phi^e}R <\infty$ for some $e$ such that $p^e\ge \overline{\nu}(R)$ then $\lind_R k<\infty$.
\end{thm}
Note that one has the following useful lemma.
\begin{lem}
\label{ld_modulo_reg_element}
Let $M$ be a finitely generated $R$-module. Let $x\in \mm^2$ be an $M$-regular element. Then
\[
\lind_R M/xM= \lind_R M+1.
\]
\end{lem}
\begin{proof}
Denote $\lind_R M=\ell$. Let $G$ be the minimal free resolution of $M$ over $R$. The morphism $\theta: G\xrightarrow{\cdot x} G$ lifts the morphism $M \xrightarrow{\cdot x} M$. Using mapping cone on the exact sequence
\[
0\to M \xrightarrow{\cdot x} M \to M/xM \to 0,
\]
then $W=G \oplus \shift G$ is minimal free resolution of $M/xM$ over $R$.

Since $x\in \mm^2$, one has $\theta(G) \subseteq \mm^2 G$. Therefore $\linp^R(W)=\linp^R G \oplus \shift\linp^R G$. Hence from $\lind_R M=\ell$, we get $H_{\ell+1}(\linp^R(W))=H_{\ell}(\linp^R G) \neq 0$ and $H_{i}(\linp^R(W))=H_{i-1}(\linp^R G) = 0$ for $i\ge \ell+2$.
So $\lind_R M/xM =\lind_R M+1$, as desired.
\end{proof}
\begin{proof}[Proof of Theorem \ref{F-Koszul-local}]
Denote $s=\overline{\nu}(R), r=\depth R$, $S=\up{\phi^e}R$ and $\nn=\up{\phi^e}\mm$. First we show that there exists a maximal $S$-regular sequence ${\bf y} \in \nn^2$ such that $k$ is isomorphic to a direct summand of the $R$-module $S/({\bf y})$. Note that $\overline{\nu}(S)=\overline{\nu}(R)=s$ since $S$ has the same ring structure as $R$. So we can choose a maximal $S$-regular sequence $\bsy=y_1,\ldots,y_r\in \nn^2$ such that $H^0_{\nn}(S/(\bsy))\cap \nn^s(S/(\bsy))=0$. Denote $\overline{S}=S/(\bsy)$. Consider the composition map $\tau: H^0_{\nn}(\overline{S})\to \overline{S} \to \overline{S}/\phi^e(\mm)\overline{S}$. Since $p^e\ge s$, we have $\phi^e(\mm)\overline{S} \subseteq \nn^s \overline{S}$ and hence $\tau$ is injective. Now $\phi^e$ induces a finite map 
$$
k=R/\mm\to \overline{S}/\phi^e(\overline{\mm})\overline{S},
$$
hence $H^0_{\nn}(\overline{S})$ is also a $k$-vector space. In particular, $\tau$ is a splitting map of $k$-vector spaces, which in turn implies that $H^0_{\nn}(\overline{S})\to \overline{S}$ is also splitting. But clearly $H^0_{\nn}(\overline{S})\neq 0$ since $\depth \overline{S}=0$. Therefore we find a copy of $k$ which is a direct summand of $\overline{S}=S/(\bsy)$.

Therefore $\lind_R k \le \lind_R S/({\bf y})$. Now from Lemma \ref{ld_modulo_reg_element}, we have
\[
\lind_R S/({\bf y})=\lind_R S +\depth R.
\]
Thanks to the hypothesis, this yields $\lind_R k \le \lind_R S +\depth R <\infty$. The proof is now complete.
\end{proof}

\subsection*{Final remarks}
Finally, we introduce several questions related to the main results.
\begin{quest}
\label{strongF-Koszul}
Let $R$ be a standard graded $k$-algebra where $k$ is $F$-finite of positive characteristic $p$. Is it true that if for some $e>0$, $\reg_R R^{(q)} <\infty$, then $R$ is Koszul?
\end{quest}
Note that the analog of this question for regular rings is not true, that is, it can happen that $\projdim_R R^{(q)}<\infty$ for every $e>0$, but $R$ is not regular. For example, let $k=\Z/(p)$ and $R=k[x,y]/(xy)$. It is not hard to check that $\phi^e: R\to R^{(q)}$ is an isomorphism for every $e>0$. On the other hand, we do not know of any counterexample to Question \ref{strongF-Koszul}.

The majority of results for Frobenius of local rings are also applied more 
generally to {\em contracting endomorphisms}; see, e.g., \cite{AHIY}. We can 
also define contracting endomorphisms for an $\N$-graded ring $R$ as follows: an 
endomorphism $\psi: R\to R$ is called {\em contracting} if for every homogeneous 
element $a\in R$, the sequence $\{\psi^i(a)\}_{i\ge 1}$ converges in the 
$\mm$-adic topology of $R$. For example, the Frobenius endomorphism and more 
generally, the endomorphisms of order at least $2$ considered in this paper are 
contracting. The homomorphism $\vphi: k[x,y]\to k[x,y]$ given by 
$\vphi(x)=x^2,\vphi(y)=y^3$, where $k[x,y]$ is standard graded, is contracting 
but has no order.
\begin{quest}
\label{contracting_endo}
Is it possible to define regularity for complexes over contracting homomorphisms and generalize Theorem \ref{orderge2} to a corresponding statement in that general setting?
\end{quest}

For local rings, we wonder if the following improvement of Theorem \ref{F-Koszul-local} is true.
\begin{quest}
Let $(R,\mm,k)$ be an $F$-finite local ring of characteristic $p>0$. Is it true that whenever $\lind_R \up{\phi^e}R<\infty$ for some $e>0$, then $\lind_R k<\infty$?
\end{quest}

\end{document}